\documentclass[runningheads]{llncs}

\newcommand{\refsec}[1] {Section~\ref{#1}}
\newcommand{\reffig}[1] {Figure~\ref{#1}}

\usepackage{amsmath}
\usepackage{amssymb}
\usepackage{graphicx}
\usepackage{ucs}
\usepackage[utf8x]{inputenc}

\usepackage{subcaption}
\usepackage{tikz}
\usepackage[outline]{contour}

\usepackage[T1]{fontenc}
\usepackage{amsmath,amssymb}
\usepackage{authblk}
\usepackage{booktabs}
\usepackage{hyperref}
\hypersetup{
    colorlinks,
    hidelinks,
    citecolor=blue,
    filecolor=black,
    linkcolor=red,
    urlcolor=green,
    plainpages=false,
    pdfpagelabels=true
}

\usepackage{url}
\newcommand{\keywords}[1]{\par\addvspace\baselineskip
\noindent\keywordname\enspace\ignorespaces#1}

\newcommand{\D}[2]{ \ensuremath{ \frac{\mathrm{d} #1 }{\mathrm{d} #2 } }}

\begin{document}

\mainmatter

\title{On the Partial Analytical Solution\\of the Kirchhoff Equation}

\titlerunning{On the Partial Analytical Solution of the Kirchhoff Equation}

\author{Dominik~L.~Michels\inst{1} \and Dmitry~A.~Lyakhov\inst{2} \and Vladimir~P.~Gerdt\inst{3} \and Gerrit~A.~Sobottka\inst{4} \and Andreas~G.~Weber\inst{4}}
\institute{Computer Science Department, Stanford University, 353 Serra Mall, MC 9515, Stanford, CA 94305, USA\\
\email{michels@cs.stanford.edu}\\
\and
Radiation Gaseous Dynamics Lab, A.~V.~Luikov Heat and Mass Transfer Institute of the National Academy of Sciences of Belarus, P.~Brovka St 15, 220072 Minsk, Belarus\\
\email{lyakhovda@hmti.ac.by}\\
\and
Group of Algebraic and Quantum Computations, Joint Institute for Nuclear Research, Joliot-Curie 6, 141980 Dubna, Moscow Region, Russia\\
\email{gerdt@jinr.ru}
\and
Multimedia, Simulation and Virtual Reality Group, Institute of Computer Science II, University of Bonn, Friedrich-Ebert-Allee 144, 53113 Bonn, Germany\\
\email{sobottka@cs.uni-bonn.de}, \email{weber@cs.uni-bonn.de}}

\authorrunning{D.~L.~Michels et al.}

\toctitle{Lecture Notes in Computer Science}
\tocauthor{On the Partial Analytical Solution of the Kirchhoff Equation}

\maketitle

\begin{abstract}
We derive a combined analytical and numerical scheme to solve the (1+1)-dimensional differential Kirchhoff system. Here the object is to obtain an accurate as well as an efficient solution process. Purely numerical algorithms typically have the disadvantage that the quality of solutions decreases enormously with increasing temporal step sizes, which results from the numerical stiffness of the underlying partial differential equations. To prevent that, we apply a differential Thomas decomposition and a Lie symmetry analysis to derive explicit analytical solutions to specific parts of the Kirchhoff system. These solutions are general and depend on arbitrary functions, which we set up according to the numerical solution of the remaining parts. In contrast to a purely numerical handling, this reduces the numerical solution space and prevents the system from becoming unstable. The differential Kirchhoff equation describes the dynamic equilibrium of one-dimensional continua, i.e. slender structures like fibers. We evaluate the advantage of our method by simulating a cilia carpet.\keywords{Differential Thomas Decomposition, Kirchhoff Rods, Lie Symmetry Analysis, Partial Analytical Solutions, Partial Differential Equations, Semi-analytical Integration.}
\end{abstract}

\section{Introduction}

In general, to study algebraic properties of a polynomially nonlinear system of partial differential equations (PDEs), the system has to be completed to involution~\cite{Seiler:2010}. In particular, by such completion one can check the consistency of the system, detect arbitrariness in its analytical solution, eliminate a subset of dependent variables, and verify whether some another PDE is a differential consequence of the system, i.e. verify whether this PDE vanishes on all solutions of the system. However, algorithmically, a polynomially nonlinear PDE system may not admit completion. Instead, one can decompose such a system into finitely many involutive subsystems with disjoint sets of solutions. This is done by the Thomas decomposition~\cite{BGLHR:2012}.\footnote{For a more-general overview of the formal algorithmic elimination for PDEs, we additionally refer to \cite{Robertz:2014}.}

In the present paper we apply first the differential Thomas decomposition to the equation system which comprises four linear PDEs in two independent and six dependent variables and also two quadratically nonlinear algebraic relations in the dependent variables. This partial differential-algebraic system is a parameter-free subsystem of the (1+1)-dimensional Kirchhoff equation. In doing so, we detect the arbitrariness in the general (analytical) solution to that subsystem. Then, by applying the classical theorem about conservative vector fields~\cite{Fikhtengol'ts:66} and performing integration by parts we reduce the parameter-free subsystem to a system of three nonlinear PDEs in two independent variables that contains one arbitrary function.

As the next step, we perform the computer algebra based Lie symmetry analysis to the reduced system and  discover that this system is equivalent (modulo Lie group transformations) to another system of three nonlinear PDEs in two independent variables $(x,y)$ and two dependent variables $f$ and $g$. The last system comprises two equations of zero Gaussian curvature for the surfaces $z=f(x,y)$ and $z=g(x,y)$ and one more equation that relates $f$ and $g$. To construct the solution of the system we use the classical parametrization of a zero curvature surface~\cite{HartmanNirenberg:59}. As a result, we obtain a closed-form analytical solution to the parameter-free subsystem under consideration and prove that the  solution is general. This opens up a wide range for application of symbolic-numeric methods to the (1+1)-dimensional Kirchhoff equation. All symbolic computations related to the construction of the general analytical solution were done with Maple.

On that basis, we derive a combined analytical and numerical scheme to solve the (1+1)-dimensional Kirchhoff equation. This system of partial differential and algebraic equations describes the dynamic equilibrium of one-dimensional continua like slender structures, artificial fibers, human hair, cilia and flagella. Since their dynamical behavior act on different time-scales, numerical stiffness enters the problem and leads to huge instabilities in the case of basic purely numerical solution schemes. To prevent that, we exploit the explicit analytical expressions found for the solution to the parameter-free subsystem. In so doing, we set up the arbitrary functions in the exact solution according to the numerical solution of the Kirchhoff equation. In contrast to a purely numerical handling, this reduces the numerical solution space and prevents the Kirchhoff system from becoming unstable. The advantage is demonstrated by simulating a cilia carpet.

In this regard our specific contributions are as follows.
\begin{enumerate}
 \item First we derive a suitable formulation of the (1+1)-dimensional differential Kirchhoff equation from the Cosserat equation and study the parameter-free subsystem of the resulting system using computer algebra-based methods and software.
 \item Then we find a closed form analytical solution to the parameter-free subsystem that depends on four parametrization functions.
 \item On the basis of results 1 and 2, we construct a combined analytical and numerical scheme to solve the full (1+1)-dimensional Kirchhoff equation.
 \item Finally, we demonstrate the advantage of our method by simulating a cilia carpet and compare the results with a purely numerical handling.
\end{enumerate}

\section{Kirchhoff's Rod Theory}

For a detailed derivation of the special Cosserat theory of rods, we refer to our last year's contribution \cite{MLGSW:2015}. According to the formulation presented in there, the rods in classical Cosserat theory can undergo shear and extension. We can avoid this by setting the linear strains to $\boldsymbol{\mathsf{\nu}}:=(\nu_1,\nu_2,\nu_3)=(0, 0, 1)$ in local coordinates. Geometrically this means, that the angle between the director $\boldsymbol{d}_3$ and the tangent to the centerline, $\partial_s\boldsymbol{r}$, always remains zero (no shear) and that the tangent to the centerline always has unit length (no elongation).\footnote{For a more-general overview of the special Cosserat and Kirchhoff rod theory, we refer to \cite{Antman:95}.} We further reduce the system to two dimensions by setting $\kappa_3=\omega_3=\upsilon_3=0$. Finally we separate the resulting equations into two PDE subsystems
\begin{eqnarray}
\label{eqn:MainSystem1}
\rho A \partial_t \vec{\upsilon} &=& \partial_s \vec{n} + \vec{f},\\
\label{eqn:MainSystem2}
\rho I \partial_t \vec{\omega} &=& \partial_s \vec{m} + \mathrm{adiag}(1,-1)\,\vec{n} + \vec{l}
\end{eqnarray}
and
\begin{eqnarray}
\label{eqn:ParameterFreeSystem1}
\partial_t \vec{\kappa} &=& \partial_s \vec{\omega},\\
\label{eqn:ParameterFreeSystem2}
\partial_s \vec{\upsilon} &=& \mathrm{adiag}(1,-1)\,\vec{\omega},
\end{eqnarray}
and additionally a set of constraints
\begin{eqnarray}
\label{eqn:Constraint1}
\vec{\omega}\,\mathrm{adiag}(1,-1)\,\vec{\kappa}^\mathsf{T}&=&0,\\
\label{eqn:Constraint2}
\vec{\upsilon}\,\mathrm{adiag}(1,-1)\,\vec{\kappa}^\mathsf{T}&=&0
\end{eqnarray}
with an antidiagonal matrix $$\mathrm{adiag}(1,-1)=\begin{pmatrix}\phantom{-}0 & \phantom{-}1\\-1 & \phantom{-}0\end{pmatrix}.$$
Please note, that the equations~(\ref{eqn:ParameterFreeSystem1})--(\ref{eqn:ParameterFreeSystem2}) are parameter free, whereas equations~(\ref{eqn:MainSystem1})--(\ref{eqn:MainSystem2}) include parameters $\rho,A,I$.

\section{Derivation of Explicit Analytical Solution}
\label{sec:AnalyticalExpressions}

In this section, we derive an exact analytical solution to (\ref{eqn:ParameterFreeSystem1})--(\ref{eqn:ParameterFreeSystem2}) under the constraints (\ref{eqn:Constraint1})--(\ref{eqn:Constraint2}).

\subsection{Arbitrariness in General Solution}
\label{sec:FunctionalArbitrariness}

To determine the functional arbitrariness in the general analytical solution to  equations~(\ref{eqn:ParameterFreeSystem1})--(\ref{eqn:Constraint2}) with non-vanishing values of the dependent variables, we apply to them the differential Thomas decomposition~\cite{BGLHR:2012,Robertz:2014} by using its implementation in Maple\footnote{The Maple package {\sc DifferentialThomas} is freely available and can be downloaded via \href{http://wwwb.math.rwth-aachen.de/thomasdecomposition/}{{\tt http://wwwb.math.rwth-aachen.de/thomasdecomposition/}}.} under the choice of a degree-reverse-lexicographical ranking with
\[
s\succ t;\qquad \omega_2\succ \omega_1\succ \kappa_2\succ \kappa_1\succ \upsilon_2\succ \upsilon_1\,.
\]

Then, under the condition $\omega_1\omega_2\kappa_1\kappa_2\upsilon_1\upsilon_2\neq 0$, the decomposition algorithm outputs two involutive differential systems. One of them is given by the set of equations
\begin{eqnarray*}
&&\upsilon_1\underline{\omega_2}-\upsilon_2\omega_1=0\,, \\
&&\underline{\partial_s \omega_1}-\partial_t\kappa_1=0\,, \\
&&\upsilon_1\underline{\kappa_2}-\upsilon_2\kappa_1=0\,,\\
&& \upsilon_1^2\omega_1^2\underline{\partial_s\kappa_1}-\kappa_1^2\upsilon_1\omega_1\partial_t\upsilon_1
  -2\kappa_1\upsilon_1^2\omega_1\partial_t\kappa_1 + \kappa_1\upsilon_2\omega_1^3=0\,,\\
&&\underline{\partial_s\upsilon_2}+\omega_1=0\,,\\
&&\kappa_1\upsilon_1^2\underline{\partial_t\upsilon_2}-\kappa_1\upsilon_2\upsilon_1\partial_t\upsilon_1+
(\upsilon_2^2+\upsilon_1^2)\omega_1^2=0\,,\\
&&\upsilon_1\underline{\partial_s\upsilon_1}-\upsilon_2\omega_1=0
\end{eqnarray*}
and inequations
$$ \upsilon_1\neq 0\,,\ \upsilon_2\neq 0\,,\ \kappa_1\neq 0\,,\ \omega_2\neq 0$$
with underlined {\em leaders}, i.e., the highest ranking derivatives occurring in the equations. The second involutive system in the output contains the equation $\upsilon_1^2+\upsilon_2^2=0$ which has no real solution under our assumption of non-vanishing dependent variables.

The differential dimensional polynomial~\cite{Markus:2014} computed for the left-hand sides of the above equation system
\[
 3\binom{l+1}{l}+1
\]
shows that the general analytical solution depends on three arbitrary functions of one variable.

\begin{remark}\label{remark:Cauchy}
This functional arbitrariness can also be understood via well-posedness of the Cauchy (initial value) problem~\cite{Gerdt:2010,Seiler:2010}. For the above involutive system, the Cauchy problem has a unique
analytical solution in a vicinity of a given initial point $(s_0,t_0)$ if $\upsilon_2(s_0,t_0)$ is an arbitrary constant and $\kappa_1(s_0,t),\omega_1(s_0,t),\upsilon_1(s_0,t)$ are arbitrary functions of $t$, analytical at the point $t=t_0$.
\end{remark}

\begin{remark}\label{remark:arbitrariness}
In the below described procedure, to construct a solution to~(\ref{eqn:ParameterFreeSystem1})--(\ref{eqn:Constraint2}) we have to solve a number of intermediate partial differential equations. In doing so, it suffices to take into account the functional arbitrariness in the solutions to those equations and to neglect the arbitrariness caused by arbitrary constants.
\end{remark}

\subsection{Reduction to Two Dependent Variables}
If $D$ is an open simply connected subset of the real space $\mathbb{R}^2$ of the independent variables $(s,t)$, then the partial differential equations (\ref{eqn:ParameterFreeSystem1}) define conservative vector fields $(\kappa_1,\omega_1)$ and $(\kappa_2,\omega_2)$ on $D$.

By the classical theorem of calculus~\cite{Fikhtengol'ts:66}\footnote{The proof is also given in {\tt http://www.owlnet.rice.edu/\textasciitilde{}fjones/chap12.pdf
\,}}.  there exist functions $p_1(s,t)$ and $p_2(s,t)$ such that
\begin{equation}\label{eqsGreen1}
\kappa_1=\partial_s p_1\,,\quad \omega_1=\partial_t p_1\,,\quad \kappa_2=\partial_s p_2\,,\quad \omega_2=\partial_t p_2\,.
\end{equation}
If one takes into account these relations, then the application of the same theorem to the equations (\ref{eqn:ParameterFreeSystem2}) transforms them into the equivalent form
\begin{equation}\label{eqsGreen2}
\upsilon_1=\partial_t f\,,\quad \upsilon_2=\partial_t g\,,\quad p_1=-\partial_s g\,,\quad p_2=\partial_s f\,,
\end{equation}
where $f=f(s,t)$ and $g=g(s,t)$ are some functions.

Therefore, the parameter-free system (\ref{eqn:ParameterFreeSystem1})--(\ref{eqn:Constraint2}) is equivalent to (i.e. it has the same solution space as) the following system of two nonlinear PDEs:
\begin{eqnarray}
&& f_{st} g_{ss}-g_{st}f_{ss}=0\,,\label{eq1fg}\\
&& g_{ss} g_t + f_{ss} f_t = 0\,. \label{eq2fg}
\end{eqnarray}
Here and below we use the notation: $f_s\equiv \partial_s f$, $f_t\equiv \partial_t f$, $f_{st}\equiv \partial_s\partial_t f$, et cetera.
We look for a solution of (\ref{eqn:ParameterFreeSystem1})--(\ref{eqn:Constraint2}) such that
\begin{equation}\label{nonzero}
\omega_1\omega_2\kappa_1\kappa_2\upsilon_1\upsilon_2\neq 0\,.
\end{equation}
The equations (\ref{eq1fg})--(\ref{eq2fg}) form a linear system in $f_{ss}$ and $g_{ss}$. It has non-trivial solutions only if
\begin{equation}
g_{st}g_t+f_{st}f_t=0\label{eq3fg}
\end{equation}
what is equivalent to the equality $\omega_1\upsilon_2-\omega_2\upsilon_1=0$ that follows from (\ref{eqn:Constraint1})--(\ref{eqn:Constraint2}).

Obviously, equation (\ref{eq3fg}) admits an integration by parts which yields
\[
 (g_t)^2+(f_t)^2-h(t)=0\,,
\]
where $h(t)$ is an arbitrary function. Therefore, the system (\ref{eqn:ParameterFreeSystem1})--(\ref{eqn:Constraint2}) is equivalent to the system
\begin{eqnarray}
&& g_{ss}g_t+f_{ss}f_t=0\,,\label{eq:fg}\\
&& (g_t)^2+(f_t)^2 - h=0\,,\quad h_s=0\,.\label{eq:fgc}
\end{eqnarray}

Now we apply the Thomas decomposition algorithm in its implementation in Maple~\cite{BGLHR:2012} for the differential elimination of the function $f$ from the system (\ref{eq:fg})--(\ref{eq:fgc}) and independently the differential elimination of the function $g$ from the same system. In doing so, inequation~(\ref{nonzero}) expressed in terms of $f$ and $g$ is taken into account. As a result of each these two eliminations, equation (\ref{eq:fg}) is replaced with another equation whereas equations (\ref{eq:fgc}) are preserved by the elimination.

Consider now the PDE system which comprises two equations in (\ref{eq:fgc}) and two more equations\footnote{Note that these equations are symmetric under the interchange of $f$ with $g$.} obtained by the differential elimination of the functions $f$ and $g$ as explained above:
\begin{eqnarray}
&& g_{ss}g_t h_t+2\,h\cdot (g_{st})^2-2\,h\,g_{ss}g_{tt}=0\,,\label{eq:fgc1} \\
&& f_{ss}f_t h_t+2\,h\cdot (f_{st})^2-2\,h\,f_{ss}f_{tt}=0\,,\label{eq:fgc2}\\
&& (g_t)^2+(f_t)^2 - h=0\,,\quad h_s=0\,.\label{eq:fgc3}
\end{eqnarray}

\begin{remark}
If one fixes the ranking and applies the differential Thomas decomposition to the equation systems (\ref{eq:fg})--(\ref{eq:fgc}) and (\ref{eq:fgc1})--(\ref{eq:fgc3}) extended with inequation (\ref{nonzero}), then in both cases the same set of equations and inequations is output. This shows explicitly the equivalence of both PDE systems as having the same solution space.
\end{remark}

\subsection{Lie Symmetry Analysis}

Since PDE system (\ref{eq:fgc1})--(\ref{eq:fgc3}) is symmetric with respect to $f$ and $g$, it suffices to  apply the methods of the Lie symmetry analysis~\cite{Ibragimov:95} to (\ref{eq:fgc2})--(\ref{eq:fgc3}). As in our paper~\cite{MLGSW:2015}, to perform related symbolic computations we use the Maple package~{\sc Desolve}~\cite{CarminatiVu:00} and its routine {\em gendef} to generate the determining linear PDE system for the infinitesimal point transformations. Then we apply the Maple package {\sc Janet}\footnote{The Maple package {\sc Janet} is freely available and can be downloaded via \href{http://wwwb.math.rwth-aachen.de/Janet/}{{\tt http://wwwb.math.rwth-aachen.de/Janet/}}.}~\cite{Janet:2003} to complete the determining system to involution\footnote{The same output can also be generated using the package {\sc DifferentialThomas}.}.

As a result, the involutive determining system containing twenty-three one-term and two-term linear PDEs is easily solvable by the Maple routine {\em pdsolve}. This gives for the coefficients of the infinitesimal Lie symmetry generator
\begin{equation}
  {\cal X}=\xi_1(s,t,f,h)\,\partial_s + \xi_2(s,t,f,h)\,\partial_t + \eta_1(s,t,f,h)\,\partial_f + \eta_2(s,t,f,h)\,\partial_h \label{generator}
\end{equation}
the following expressions:
\[
 \eta_1=c_1s+c_2f+c_3\,,\ \eta_2=-\D{\phi}{t}+c_4\,,\ \xi_1=c_5s+c_6f+c_7\,,\ \xi_2=\phi\,,
\]
where $c_i$ $(1\leq i\leq 7)$ are arbitrary constants, and $\phi=\phi(t)$ is an arbitrary function.

As we are looking for the functional arbitrariness only (see Remark~\ref{remark:arbitrariness}), we put $c_i=0$. The Lie symmetry group defined by the generator~(\ref{generator}) is obtained by solving the Lie equations
\[
 \frac{ds}{d\alpha}=\xi_1\,,\quad  \frac{dt}{d\alpha}=\xi_2\,,\quad  \frac{df}{d\alpha}=\eta_1\,,\quad \frac{dh}{d\alpha}=\eta_2\,.
\]
For the independent variables this gives the one-parameter transformation group
\[
   \tilde{s}=s,\quad \tilde{t}=\Phi(t,\alpha)\,.
\]
Here $\alpha$ is a group parameter and $\Phi$ is an arbitrary function such that $\phi(t)=\partial_\alpha \Phi|_{\alpha=0}$.  Now we can formulate and prove one of the main theoretical results of the present paper.

\begin{lemma} Let $h(t)> 0$ be a given function. Then for any solution to the equation (\ref{eq:fgc2}) there exists an invertible transformation of variables $(s,t)\rightarrow (x,y)$ of the form
\begin{equation}
   x=s, \quad y=\Psi(t) \label{transformation}
\end{equation}
such that the solution is an image of a solution to the equation
\begin{equation}
  f_{xx}f_{yy}-f_{xy}^2=0\,. \label{zerocurvaturef}
\end{equation}
\end{lemma}

\begin{proof}
First, we note that the transformation (\ref{transformation}) is invertible if and only if $\Psi_t\neq 0$. Second, we apply this transformation to (\ref{zerocurvaturef}). This gives
\begin{equation}
  \frac{f_{ss}f_{tt}-(f_{st})^2}{f_{ss}f_t}=\frac{\Psi_{tt}}{\Psi_t}\,. \label{rhsF}
\end{equation}
On the other hand, equation (\ref{eq:fgc2}) can be rewritten as
\begin{equation}
  \frac{f_{ss}f_{tt}-(f_{st})^2}{f_{ss}f_t}=\frac{h_t}{2\,h}\,. \label{rhsC}
\end{equation}
We see that at $(\Psi_t)^2=h$ the right-hand sides of (\ref{rhsF}) and (\ref{rhsC}) coincide.\hfill$\Box$
\end{proof}
This lemma immediately implies the following proposition.

\begin{proposition}
The PDE system (\ref{eq:fgc1})--(\ref{eq:fgc3}) can be obtained from the system
\begin{eqnarray}
&& f_{xx}f_{yy}-(f_{xy})^2=0\,, \label{eq:fGaussian}\\
&& g_{xx}g_{yy}-(g_{xy})^2=0\,,\label{eq:gGaussian}\\
&& (f_y)^2+(g_y)^2-1=0\label{eq:fgConstraint}
\end{eqnarray}
by the transformation (\ref{transformation}) of the independent variables with $h=\Psi_t^2$.
\end{proposition}
Thus, we reduced the parameter-free subsystem (\ref{eqn:ParameterFreeSystem1})--(\ref{eqn:Constraint2}) of the Kirchhoff equation to the equivalent system (\ref{eq:fGaussian})--(\ref{eq:fgConstraint}). In doing so, the independent variables of the latter system are related to the independent variables of the former one by equalities (\ref{transformation}). The advantages of the reduced system are:
\begin{enumerate}
\item There are only two dependent variables.
\item The system is symmetric with respect to the dependent variables.
\item They are separated in (\ref{eq:fGaussian}) and (\ref{eq:gGaussian}).
\item The differential equation of the form (\ref{eq:fGaussian}) is very well known in the theory of analytical surfaces and deeply studied in the literature~\cite{Krivoshapko:2015}. It means that the surface $z=f(x,y)$ has zero Gaussian curvature. Such surfaces are locally isometric to a plane and admit a parametrization constructed in the classical paper~\cite{HartmanNirenberg:59}. In the following subsection we use this parametrization to derive the general analytical solution to (\ref{eqn:ParameterFreeSystem1})--(\ref{eqn:Constraint2}).
\end{enumerate}

\subsection{Parametrized Analytical Solution}

In~\cite{HartmanNirenberg:59} a parametric form of a general analytical solution to an isolated zero curvature equation was found. For equation~(\ref{eq:fGaussian}) this parametrization reads
\begin{equation}
 f=a_1(u)v+b_1(u)\,,\quad x=a_2(u)v+b_2(u)\,,\quad y=a_3(u)v+b_3(u)\,. \label{f-parametrization}
\end{equation}
Here $a_i$ and $b_i$ are analytical functions of the parameter $u$.

\begin{remark}\label{rem:parametrization}
In general, the function $g$ may have the parameterizations of $x$ and $y$ such that $a_i(u),b_i(u)$ $(i=2,3)$ are different from that in (\ref{f-parametrization}). We assume, however, that both functions $f$ and $g$ have the same parametrizations of $x$ and $y$ so that $a_2\neq 0$.
\end{remark}

Our assumption allows to express the parameter $v$ linearly in terms of $x$ and leads to a rather compact parametrization of functions $f,g$ of the form
\begin{equation} \label{parametrization}
f=A_1(u)\,x+B_1(u)\,,\quad y=A(u)\,x+B(u)\,,\quad g=M(u)\,x+N(u)
\end{equation}
with six functions. However, these functions are not arbitrary. By taking the first and second order derivatives of both sides in the equality
\[
 f\left(x,A(u)\,x+B(u)\right)=A_1(u)x+B_1(u)
\]
with respect to $x$ and $u$ we obtain a linear equation system in $f_x,f_y,f_{xx},f_{yy},f_{xy}$. Its solution  allows to express the partial derivatives $f_{xx},f_{yy},f_{xy}$ in terms of functions $A(u),A_1(u),B(u),B_1(u)$ and their first derivatives. Then, the substitution of these expressions into equation (\ref{eq:fGaussian}) shows that this equation is satisfied if
\begin{equation}
(A_1)_u=P(u)\,A_u\,,\quad (B_1)_u=P(u)\,B_u\,, \label{eq:forf}
\end{equation}
where $P(u)$ can be any function. Similarly, equation (\ref{eq:gGaussian}) is rewritten as two differential equalities
\begin{equation}
M_u=Q(u)\,A_u\,,\quad N_u=Q(u)\,B_u \label{eq:forg}
\end{equation}
with unspecified $Q(u)$. Next, equation~(\ref{eq:fgConstraint}) implies the algebraic equality $Q(u)^2+P(u)^2=1$ or equivalently
\begin{equation}
    Q(u)=\sin(C(u))\,,\qquad P(u)=\cos(C(u)) \label{eq:forfg}
\end{equation}
with arbitrary $C(u)$.

As the next step, we apply the following transformation of the independent variables, which are inverse to (\ref{transformation}):
\[
   s=x\,,\quad t=F(y)
\]
with $F=\Psi^{-1}$. For this purpose, we use the relations
\begin{eqnarray*}
&& f\left(s,F\left(A(u)\,s+B(u)\right)\right)=A_1(u)s+B_1(u)\,,\\
&& g\left(s,F\left(A(u)\,s+B(u)\right)\right)=M(u)s+N(u)
\end{eqnarray*}
together with (\ref{eq:forf}), (\ref{eq:forg}), and (\ref{eq:forfg})
to compute all partial derivatives of the functions $f$ and $g$ with respect to $s$ and $t$ that occur in (\ref{eqsGreen1}) and (\ref{eqsGreen2}).

This gives the analytical solution:
\begin{eqnarray}
\vec{\kappa} &=& -\frac{A^2(u)C'(u)}{A'(u)s+B'(u)}\begin{pmatrix}\cos(C(u))\\\sin(C(u))\end{pmatrix},\nonumber \\
\vec{\omega} &=& \frac{A(u)C'(u)}{F'(A(u)s+B(u))(A'(u)s+B'(u)))}\begin{pmatrix}\cos(C(u))\\\sin(C(u))\end{pmatrix}, \label{eqn:solution}\\
\vec{\upsilon} &=& \frac{1}{F'(A(u)s+B(u))}\begin{pmatrix}\cos(C(u))\\\sin(C(u))\end{pmatrix}\,,\quad t=F\left(A(u)\,s+B(u)\right)\,. \nonumber
\end{eqnarray}
Here $A(u),B(u),C(u)$ are arbitrary analytical functions of one variable, $F$ is an arbitrary analytical function of the argument $A(u)\,s + B(u)$, and the prime mark denotes differentiation of a function with respect of its argument.

Substituting (\ref{eqn:solution}) in equations (\ref{eqn:ParameterFreeSystem1})--(\ref{eqn:Constraint2}) shows that these equations are satisfied for any functions $A,B,C,F$. In \refsec{sec:FunctionalArbitrariness} we saw that the general analytical solution depends on three arbitrary functions, each of one variable. The following theorem shows that the functional arbitrariness in (\ref{eqn:solution}) is superfluous.

\begin{theorem}
Set $B(u)=u$ in (\ref{eqn:solution}). Then the analytical solution
\begin{eqnarray}
\vec{\kappa} &=& -\frac{A^2(u)C'(u)}{A'(u)s+1}\begin{pmatrix}\cos(C(u))\\\sin(C(u))\end{pmatrix}\,,\nonumber \\
\vec{\omega} &=& \frac{A(u)C'(u)}{F'(A(u)s+u)(A'(u)s+1)))}\begin{pmatrix}\cos(C(u))\\\sin(C(u))\end{pmatrix}, \label{eqn:gensolution}\\
\vec{\upsilon} &=& \frac{1}{F'(A(u)s+u)}\begin{pmatrix}\cos(C(u))\\\sin(C(u))\end{pmatrix}\,,\quad t=F\left(A(u)\,s+u\right)\,. \nonumber
\end{eqnarray}
to (\ref{eqn:ParameterFreeSystem1})--(\ref{eqn:Constraint2}) is general.
\end{theorem}

\begin{proof}
Assume without loss of generality that the initial point is $(s_0,t_0)=(0,0)$ and pose the Cauchy problem for (\ref{eqn:ParameterFreeSystem1})--(\ref{eqn:Constraint2}) in a vicinity of this point. Then the well-posed Cauchy problem reads (see Remark \ref{remark:Cauchy}):
\begin{equation}
 \upsilon_1(0,t)=f_1(t)\,, \quad \omega_1(0,t)=f_2(t)\,,\quad \kappa_1(0,t)=f_3(t)\,,\quad \upsilon_2(0,0)=C_1\,,\label{CauchyProblem}
\end{equation}
where $f_1,f_2,f_3$ are analytic functions and $C_1$ is an arbitrary constant. Substitution of (\ref{eqn:gensolution}) into (\ref{CauchyProblem}) implies equations
\begin{eqnarray}
&& \frac{\cos(C(u))}{F'(u)} = f_1(F(u))\,,\nonumber \\
&& \frac{C'(u)A(u)\cos(C(u))}{F'(u)} = f_2(F(u))\,,\nonumber \\
&& -C'(u)A^2(u)\cos(C(u)) = f_3(F(u))\,, \label{eqsforf}\\
&& \frac{\sin(C(0))}{F'(0)}=C_1\,,\nonumber \\
&& F(0)=0\,. \nonumber
\end{eqnarray}
Now we aim to show that for any predetermined functions satisfying
$$ f_1(t)\neq 0\,,\quad f_2(t)\neq 0\,,\quad f_3(t)\neq 0\,,\quad \upsilon_2(0,0)=C_1\neq 0$$
it is always possible to match up functions $C(u),A(u),F(u)$ such that in a vicinity of $(0,0)$ solution (\ref{eqn:gensolution}) satisfy the initial data (\ref{CauchyProblem}). The squared second equation in (\ref{eqsforf}) divided by the third equation gives the equality
$$ -\frac{C'(u)\cos(C(u))}{(F'(u))^2}=\frac{f^2_2(F(u))}{f_3(F(u))}$$
which together with the first and third equations in (\ref{eqsforf}) implies
\begin{eqnarray}
&& C'(u)=-\frac{f_2^2(F(u))\cos(C(u))}{f_3(F(u))f_1^2(F(u))}\,,\quad C(0)=\arctan\left(\frac{C_1}{f_1(0)}\right)\,,\nonumber \\
&& F'(u)=\frac{\cos(C(u))}{f_1(F(u))}\,,\quad F(0)=0\,,\label{NewCauchy}\\
&& A(u)=-\frac{f_3(F(u))f_1(F(u))}{f_2(F(u))\cos(C(u))}\,. \nonumber
\end{eqnarray}
If $f_1(0)\neq 0$, $f_2(0)\neq 0$, $f_3(0)\neq 0$, then in some vicinity of zero the problem (\ref{NewCauchy}) has the unique solution. Indeed, the first two equations in (\ref{NewCauchy}) define a Cauchy problem with regular right-hand sides, and the last equation sets function $A(u)$ explicitly. In doing so, $C(0)\in (\pi/2,-\pi/2)$ and, hence, $\cos(C(0))\neq 0$. In a vicinity of $u=0$ all factors occurring in $C'(u),F'(u),A(u)$ do not vanish, and the unique solution of this system satisfies (\ref{eqsforf}) identically. The last is easily verified by substitution of (\ref{NewCauchy}) in (\ref{eqsforf}).
\hfill$\Box$
\end{proof}

\section{Numerical Experiments}
\label{sec:Experiments}

The resulting analytical solutions (\ref{eqn:gensolution}) for the system (\ref{eqn:ParameterFreeSystem1})--(\ref{eqn:ParameterFreeSystem2}) under the constraints (\ref{eqn:Constraint1})--(\ref{eqn:Constraint2}) contain four parameterization functions, which can be determined by the numerical integration of the equations~(\ref{eqn:MainSystem1})--(\ref{eqn:MainSystem2}). The substitution of the equations~(\ref{eqn:gensolution}) into the equations~(\ref{eqn:MainSystem1})--(\ref{eqn:MainSystem2}), the replacement of the spatial derivatives with central differences, and the replacement of the temporal derivatives according to the numerical scheme of a forward Euler integrator, leads to an explicit expression.\footnote{We do not explicitly write out the resulting equations here for brevity.} Iterating over this recurrence equation allows for the simulation of the dynamics of a Kirchhoff rod.

We present this in the context of the simulation of the beating pattern of a cilium. Such kind of scenarios are of interest in the context of simulations in biology and biophysics, since this kind of patterns occur widely in nature. For example cilia carpets in the interior of the lung are responsible for the mucus transport, see \cite{FB:1986}. The results are shown in \reffig{fig:BeatingPattern}. Compared to a purely numerical handling using a discretization of the equations (\ref{eqn:MainSystem1})--(\ref{eqn:ParameterFreeSystem2}) in a similar way, the step size can be increased by three orders of magnitude, which leads to an acceleration of two orders of magnitude. This is of special interest for complex systems like cilia carpets, in which multiple cilia are beating in parallel with phase differences in order to produce the appearance of a wave. This is illustrated in \reffig{fig:BeatingPattern}.

\begin{figure}
\centering
\includegraphics[width=1.0\textwidth]{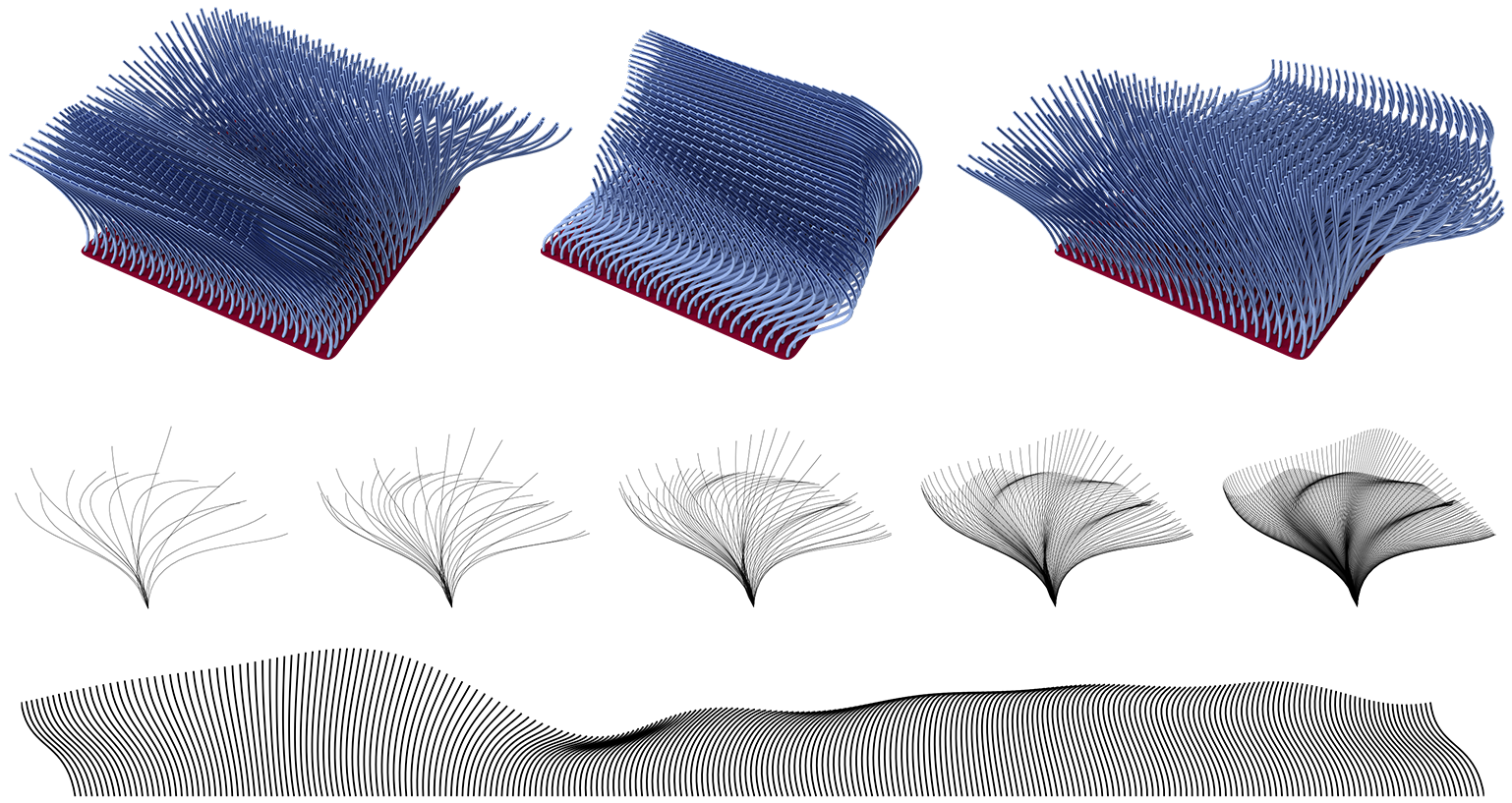}
\caption{Simulation of a cilia carpet (top) composed of multiple cilia beating in a metachronal rhythm (middle). This produces the appearance of a wave (below).}
\label{fig:BeatingPattern}
\end{figure}

\section{Conclusion}

In this contribution, we have studied the two-dimensional Kirchhoff system. To prevent from numerical problems which result from the numerical stiffness of the underlying differential equations, we developed a combined, partially symbolic and partially numeric, integration scheme. Within the symbolic part we constructed an explicit analytical solution to the parameter-free part of Kirchhoff system and proved that this solution is general. The application of the analytical solution prevents from numerical instabilities and allows for highly accurate and efficient simulations. This was demonstrated for the scenario of a cilia carpet, which could be efficiently simulated two orders of magnitude faster compared to a purely numerical handling.

\section*{Acknowledgements}

This work was partially supported by the Max Planck Center for Visual Computing and Communication (D.L.M.) as well as by the grant No.~13-01-00668 from the Russian Foundation for Basic Research (V.P.G.). The authors thank Robert Bryant for useful comments on the solution of PDE system (\ref{eq:fGaussian})-(\ref{eq:fgConstraint}) and the anonymous referees for their remarks and suggestions.



\end{document}